\documentclass{amsart}
\input epsf

\usepackage{amsfonts,amsthm,amsmath,amssymb,latexsym}
\usepackage{graphicx,color}
\usepackage[all]{xy}

\begin{document}

\author{Dragomir \v Sari\' c}
\thanks{This research was partially supported by National Science Foundation grant DMS 1102440.}

\address{Department of Mathematics, Queens College of CUNY,
65-30 Kissena Blvd., Flushing, NY 11367}
\email{Dragomir.Saric@qc.cuny.edu}

\address{Mathematics PhD. Program, The CUNY Graduate Center, 365 Fifth Avenue, New York, NY 10016-4309}

\theoremstyle{definition}

 \newtheorem{definition}{Definition}[section]
 \newtheorem{remark}[definition]{Remark}
 \newtheorem{example}[definition]{Example}

\newtheorem*{notation}{Notation}

\theoremstyle{plain}

 \newtheorem{proposition}[definition]{Proposition}
 \newtheorem{theorem}[definition]{Theorem}
 \newtheorem{corollary}[definition]{Corollary}
 \newtheorem{lemma}[definition]{Lemma}

\newcommand{\eps}{\varepsilon}
\newcommand{\G}{\Gamma}
\newcommand{\g}{\gamma}
\newcommand{\D}{\Delta}
\renewcommand{\d}{\delta}
\newcommand{\dist}{\mathrm{dist}}
\newcommand{\m}{\mathrm{mod}}

\title[Upper bounded pants decomposition]{Fenchel-Nielsen coordinates on upper bounded pants decompositions}


\date{\today}

\maketitle

\begin{abstract}
Let $X_0$ be an infinite type hyperbolic surface (whose boundary components, if any, are closed geodesics or punctures) which has an upper bounded pants decomposition. The length spectrum Teichm\"uller space $T_{ls}(X_0)$ consists of all surfaces $X$ homeomorphic to $X_0$ such that the ratios of the corresponding simple closed geodesics are uniformly bounded from below and from above. Alessandrini, Liu, Papadopoulos and Su \cite{ALPS} described the Fenchel-Nielsen coordinates for $T_{ls}(X_0)$ and using these coordinates they proved that $T_{ls}(X_0)$ is path connected. We use the Fenchel-Nielsen coordinates for $T_{ls}(X_0)$ to induce a locally biLipschitz homeomorphism between $l^{\infty}$ and $T_{ls}(X_0)$ (which extends analogous results by Fletcher \cite{Fletcher} and by Allessandrini, Liu, Papadopoulos, Su and Sun \cite{ALPSS} for the unreduced and the reduced $T_{qc}(X_0)$). Consequently, $T_{ls}(X_0)$ is contractible. We also characterize the closure in the length spectrum metric of the quasiconformal Teichm\"uller space $T_{qc}(X_0)$ in $T_{ls}(X_0)$.
\end{abstract}

\section{Introduction}

Let $X_0$ be a complete hyperbolic surface of infinite type whose boundary components, if any, are closed geodesics or punctures. Assume that there exists a pants decomposition $\mathcal{P}=\{\alpha_n\}$ of $X_0$ by simple closed geodesics such that their lengths are bounded from the above by a fixed constant $M_0>0$. We say that such a pants decomposition is {\it upper bounded}. By \cite{ALPSS}, any other pants decomposition of $X_0$ by simple closed curves can be straightened to a pants decomposition by simple closed geodesics. 

The {\it quasiconformal} Teichm\"uller space $T_{qc}(X_0)$ consists of all quasiconformal maps $f:X_0\to X$ up to post composition by isometries and up to bounded homotopies which setwise fix boundary components of $X$. Note that bounded homotopies do not fix boundary geodesics pointwise since the distance between any two points (on a boundary geodesic) is finite. Thus $T_{qc}(X_0)$ is the reduced quasiconformal Teichm\"uller space of the surface $X_0$.

The {\it length spectrum} Teichm\"uller space $T_{ls}(X_0)$ 
consists of all homeomorphisms $h:X_0\to X$ up to post compositions by isometries and up to bounded homotopies that setwise preserve the boundary components of $X$ such that
$$
L(X_0,X):=\sup_{\beta}\max\{ \frac{l_{\beta}(X)}{l_{\beta}(X_0)},\frac{l_{\beta}(X_0)}{l_{\beta}(X)}\}<\infty ,
$$
where the supremum is over all simple closed curve $\beta$ on $X_0$,
and where $l_{\beta}(X),l_{\beta}(X_0)$ are the lengths of the geodesic representatives of $\beta$ on $X,X_0$, respectively.  Note that $T_{qc}(X_0)\subset T_{ls}(X_0)$ because for each simple closed geodesic $\beta$ on $X_0$ and a $K$-quasiconformal map $f:X_0\to X$ we have (cf. Wolpert \cite{Wolpert}) 
$$
\frac{1}{K}l_{\beta}(X_0)\leq l_{\beta}(X)\leq Kl_{\beta}(X_0).
$$
The {\it length spectrum metric} $d_{ls}$ is defined by
$$
d_{ls}(X,Y)=\frac{1}{2}\log L(X,Y)
$$
for $X,Y\in T_{ls}(X_0)$. Shiga \cite{Shiga} was the first to study the length spectrum metric on quasiconformal Teichm\"uller spaces of  surfaces of infinite type and he proved that $d_{ls}$ in general is not complete on $T_{qc}(X_0)$. This implies that $T_{qc}(X_0)$ could be a proper subset of $T_{ls}(X_0)$.

For a fixed upper bounded pants decomposition $\mathcal{P}=\{\alpha_n\}$ on $X_0$, the assignment of Fenchel-Nielsen coordinates $\{ (l_{\alpha_n}(X),t_{\alpha_n}(X))\}$ to each $X\in T_{qc}(X_0)$ or $X\in T_{ls}(X_0)$ completely determines the (marked) surface $X$. Alessandrini, Liu, Papadopoulus, Su and Sun \cite{ALPSS} proved that the Fenchel-Nielsen coordinates for $T_{qc}(X_0)$ satisfy
$\sup_n |\log\frac{l_{\alpha_n}(X)}{l_{\alpha_n}(X_0)}|<\infty$ and $\sup_n |t_{\alpha_n}(X)-t_{\alpha_n}(X_0)|<\infty$, and that the map from $T_{qc}(X_0)$ to the Fenchel-Nielsen coordinates is a locally biLipschitz homeomorphism onto $l^{\infty}$. Alessandrini, Liu, Papadopuolus and Su \cite{ALPS} proved that the Fenchel-Nielsen coordinates for $X\in T_{ls}(X_0)$ satisfy 
$$
\sup_{\alpha_n\in\mathcal{P}}|\log\frac{l_{\alpha_n}(X)}{l_{\alpha_n}(X_0)}|<\infty
$$ 
and
$$
\sup_{\alpha_n\in\mathcal{P}}\frac{|t_{\alpha_n}(X)-t_{\alpha_n}(X_0)|}{\max\{ 1,|\log l_{\alpha_n}(X_0)|\}} <\infty.
$$
Moreover, they proved that if $X_0$ (equipped with an upper bounded pants decomposition) contains a sequence of simple closed geodesics whose lengths go to $0$ then $T_{qc}(X_0)\subsetneq T_{ls}(X_0)$ and that $T_{qc}(X_0)$ is nowhere dense in $T_{ls}(X_0)$ (cf. \cite{ALPS}).

Let $F:T_{ls}(X_0)\to l^{\infty}$, called the {\it normalized Fenchel-Nielsen map}, be defined by
\begin{equation*}
F(X)=\Big{\{} (\log\frac{l_{\alpha_n}(X)}{l_{\alpha_n}(X_0)},\frac{t_{\alpha_n}(X)-t_{\alpha_n}(X_0)}{\max\{ 1,|\log l_{\alpha_n}(X_0)|\}})\Big{\}}_n
\end{equation*}
where if $\alpha_n$ is a boundary geodesic of $X_0$ then we take only the length component of the Fenchel-Nielsen coordinates. 

By \cite{ALPSS}, the Fenchel-Nielsen coordinates give a locally biLipschitz homeomorphism between the (reduced) quasiconformal Teichm\"uller space $T_{qc}(X_0)$ and $l^{\infty}$ which implies that $T_{qc}(X_0)$ is contractible. Fletcher \cite{Fletcher} used complex analytic methods to prove that the unreduced quasiconformal Teichm\"uller space is locally biLipschitz to $l^{\infty}$.
Our main result is (cf. \S 2, Theorem \ref{thm:main} and Corollary \ref{cor:contractible}):

\vskip .2 cm

\paragraph{\bf Theorem 1.} {\it The normalized Fenchel-Nielsen map
$$
F:T_{ls}(X_0)\to l^{\infty}
$$
is a locally biLipschitz homeomorphism. 

In particular, the length spectrum Teichm\"uller space $T_{ls}(X_0)$ is contractible.}

\vskip .2 cm

Thus \cite{Fletcher}, \cite{ALPSS} and Theorem 1 imply that the unreduced quasiconformal Teichm\"uller space, the reduced quasiconformal Teichm\"uller space and the length spectrum Teichm\"uller space are locally biLipschitz to $l^{\infty}$ and thus to each other.

A problem of characterizing the closure of $T_{qc}(X_0)$ inside $T_{ls}(X_0)$ for the length spectrum metric $d_{ls}$ is raised in \cite{ALPS}. We use the Fenchel-Nielsen coordinates to characterize the closure $\overline{T_{qc}(X_0)}$ of $T_{qc}(X_0)$.

\vskip .2 cm

\paragraph{\bf Theorem 2.} {\it Let $X_0$ be an infinite type hyperbolic surface with an upper bounded pants decomposition $\mathcal{P}=\{\alpha_n\}$. Then $X\in \overline{T_{qc}(X_0)}$ if and only if
$$
\sup_{\alpha_n\in\mathcal{P}}|\log\frac{l_{\alpha_n}(X)}{l_{\alpha_n}(X_0)}|<\infty
$$ 
and
$$
|t_{\alpha_n}(X)-t_{\alpha_n}(X_0)|=o(|\log l_{\alpha_n}(X_0)|)
$$
as $|\log l_{\alpha_n}(X_0)|\to \infty$.
}

\section{The Fenchel-Nielsen coordinates}

We prove that the normalized Fenchel-Nilesen map is a localy biLipschitz homeomorphism onto $l^{\infty}$.

\begin{theorem}
\label{thm:main}
Let $X_0$  be an infinite type complete hyperbolic surface equipped with an upper bounded geodesic pants decomposition $\mathcal{P} =\{ \alpha_n\}_{n\in\mathbb{N}}$.  The normalized Fenchel-Nielsen map
\begin{equation}
\label{eq:fn-norm}
F(X)=\Big{\{} \Big{(}\log \frac{l_{\alpha_n}(X)}{l_{\alpha_n}(X_0)},\frac{t_{\alpha_n}(X)-t_{\alpha_n}(X_0)}{\max\{ 1,|\log l_{\alpha_n}(X_0)|\}}\Big{)}\Big{\}}_{n\in\mathbb{N}}
\end{equation}
for $X\in T_{ls}(X_0)$, induces a locally biLipschitz surjective homeomorphism
$$
F:T_{ls}(X_0)\to l^{\infty}.
$$
\end{theorem}

\begin{proof} Let $M_0$ be such that $l_{\alpha_n}(X_0)\leq M_0$ for each $\alpha_n\in\mathcal{P}$.

\vskip .1 cm

\noindent {\bf Step I:}
We establish that $F(T_{ls}(X_0))\subset l^{\infty}$ which is already proved in \cite{ALPS}. We give another proof in order to facilitate the rest of the argument. By the definition, $X\in T_{ls}(X_0)$ if there is $M>0$ such that
$$
|\log\frac{l_{\gamma}(X)}{l_{\gamma}(X_0)}|\leq M
$$ 
for each simple closed curve $\gamma\in\mathcal{C}$ on $X_0$. In particular $\{ \log \frac{l_{\alpha_n}(X)}{l_{\alpha_n}(X_0)}\}_{n\in\mathbb{N}}$ is a bounded sequence. 

It remains to bound the twists. The choice of the twists of $X_0$ on $\alpha_n$ are determined up to integer multiples of $l_{\alpha_n}(X_0)$. Without loss of generality, we normalize them such that, for each $n\in\mathbb{N}$, $$0\leq t_{\alpha_n}(X_0)< l_{\alpha_n}(X_0).$$ 
Given this normalization, it is enough to prove that 
$$
|t_{\alpha_n}(X)|/\max\{ 1,|\log l_{\alpha_n}(X_0)|\}
$$
is bounded uniformly in $n\in\mathbb{N}$.

Using \cite{Bishop}, there exists a surface $X'$ which is $K$-quasiconformal to $X_0$ such that $l_{\alpha_n}(X')=l_{\alpha_n}(X)$ for all $n\in\mathbb{N}$, where $K=K(M)$ (cf. \cite{ALPS}). The $K$-quasiconformal map  $f:X_0\to X'$ maps each pair of pants $P\in\mathcal{P}$ of $X_0$ onto a geodesic pair of pants of $X'$ such that on each boundary geodesic the map is affine. Divide each geodesic pair of pants into two right angled hexagons by three {\it seams}, namely three geodesic arcs connecting pairs of boundary curves and orthogonal to them. Each hexagon contains half of each boundary geodesic of the pair of pants which are called {\it a-sides}. The other three sides of the hexagon which are seams are called {\it b-sides}. The two hexagons are glued along their b-sides to obtain the pair of pants and the pairwise union of their a-sides forms three geodesic boundaries of the pair of pants. The map $f:X_0\to X'$ maps a-sides of each hexagon of each pair of pants of $X_0$ onto a-sides of hexagons of the corresponding pair of pants of $X$ and it is affine on the a-sides. Note that a single $\alpha_n\in\mathcal{P}$ is on the boundary of two pairs of pants $P_1^0$ and $P_2^0$ of $\mathcal{P}$ which implies that $\alpha_n$ is divided into a-sides with respect to both $P_1^0$ and $P_2^0$. The two divisions of $\alpha_n$ into a-sides do not match in general and the distance between the endpoints of two a-sides coming from two pairs of pants is the twist parameter $t_{\alpha_{n}}(X_0)$ of $X_0$ at the closed geodesic $\alpha_n$ for the pants decomposition $\mathcal{P}$. The map $f:X_0\to X'$ is affine on each $\alpha_n$, it maps the foots of the seams of the pair of pants $P_i^0$ to the foots of the corresponding seams of $f(P_i^0)=P_i'$ for $i=1,2$ and it does not introduce any full twisting along $\alpha_n$ by its construction (cf. \cite{ALPS}, \cite{Bishop}).  Then for $K=K(M)$ we have
$$
t_{\alpha_n}(X')=\frac{l_{\alpha_n}(X')}{l_{\alpha_n}(X_0)}t_{\alpha_n}(X_0)\leq K t_{\alpha_n}(X_0).
$$

Let $t_n=t_{\alpha_n}(X)-t_{\alpha_n}(X')$. Then $X$ is obtained by a(n infinite) multi twist on $X'$ along the family $\mathcal{P}=\{\alpha_n\}$ by the amount $\{ t_n\}$. It is enough to prove that $\frac{|t_n|}{\max\{ 1,|\log l_{\alpha_n}(X_0)\}}$ is bounded in terms of $d_{ls}(X_0,X)$ because $|t_{\alpha_n}(X')|\leq KM_0$. This is proved in \cite{ALPS} using results of Minsky \cite{Min} and Choi-Rafi \cite{CR}. We give a more direct proof of this result below.
Fix a cuff $\alpha_n$ and let $P_1'=f(P_1^0)$ and $P_2'=f(P_2^0)$ be two geodesic pairs of pants with common boundary $\alpha_n$. Either $P_1'\neq P_2'$ or $P_1'=P_2'$ and we divide the argument into these two cases.

{\bf Case 1.}
Assume that $P_1'\neq P_2'$. There exists a unique geodesic arc $\gamma^i_n\subset P_i'$, for $i=1,2$, which starts and  ends at $\alpha_n$ that is orthogonal to $\alpha_n$ at both of its endpoints. Let $\beta_n$ be a closed curve on $X'$ obtained by concatenating $\gamma^1_n$ followed by an arc of $\alpha_n$ from an endpoint of $\gamma^1_n$ to an endpoint of $\gamma^2_n$ in the direction of the left twist along $\alpha_n$ followed by $\gamma^2_n$ followed by an arc of $\alpha_n$ connecting other two endpoints of $\gamma^1_n$ and $\gamma^2_n$ in the direction of the right twist (cf. Figure 1). 

\begin{figure}
\centering
\includegraphics[scale=0.7]{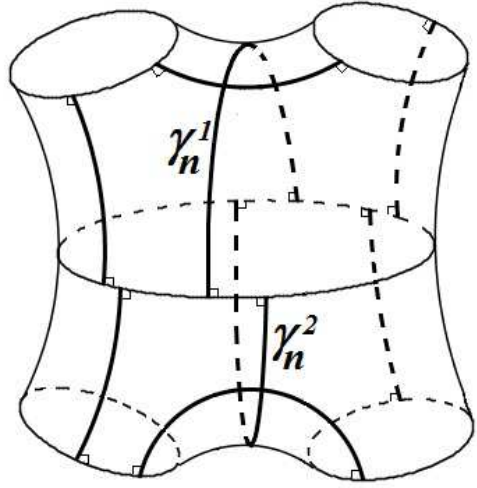}
\caption{The curve $\beta_n$.}
\end{figure}

We will give an upper bound for $t_{\alpha_n}(X)$ in terms of $l_{\beta_n}(X)$. 
Fix three consecutive lifts $\tilde{\alpha}_n^j$, for $j=1,2,3$, of $\alpha_n$ under the universal covering $\pi :\mathbb{H}^2\to X$. Let $\tilde{\beta}_n^{*}$ be the lift of the geodesic representative $\beta_n^{*}$ of $\beta_n$ that intersects  $\tilde{\alpha}_n^j$, for $j=1,2,3$. Moreover, let $\tilde{\gamma}^i_n$ be the lift of $\gamma^i_n$ that connects $\tilde{\alpha}^{i}_n$ and $\tilde{\alpha}^{i+1}_n$ (cf. Figure 2). Let $a_1=\tilde{\gamma}_n^1\cap\tilde{\alpha}_n^2$ and $a_2=\tilde{\gamma}_n^2\cap\tilde{\alpha}_n^2$, and $b=\tilde{\beta}_n^{*}\cap\tilde{\alpha}_n^2$. The lengths satisfy $l_{\gamma_n^i}(X')=l_{\gamma_n^i}(X)$ because $X$ is obtained from $X'$ by a multi twist along $\{\alpha_n\}$. We either have $d_{hyp}(a_1,b)\geq |t_{\alpha_n}(X)|/2$ or $d_{hyp}(a_2,b)\geq |t_{\alpha_n}(X)|/2$. Consider the case that 
\begin{equation}
\label{eq:*}
d_{hyp}(a_2,b)\geq |t_{\alpha_n}(X)|/2
\end{equation} 
and the other case is analogous. Let $c_2=\tilde{\gamma}_n^2\cap \tilde{\alpha}_n^3$ and let $c_1\in\tilde{\alpha}_n^3$ be the foot of the orthogonal from $b$ to $\tilde{\alpha}_n^3$. Consider the quadrilateral with vertices $b$, $c_1$, $c_2$ and $a_2$ (cf. Figure 2). We get
\begin{equation}
\label{eq:**}
\sinh d_{hyp}(b,c_1)=\sinh l_{\tilde{\gamma}_n^2}(X)\cosh d_{hyp}(b,a_2).
\end{equation}

\begin{figure}
\centering
\includegraphics[scale=0.7]{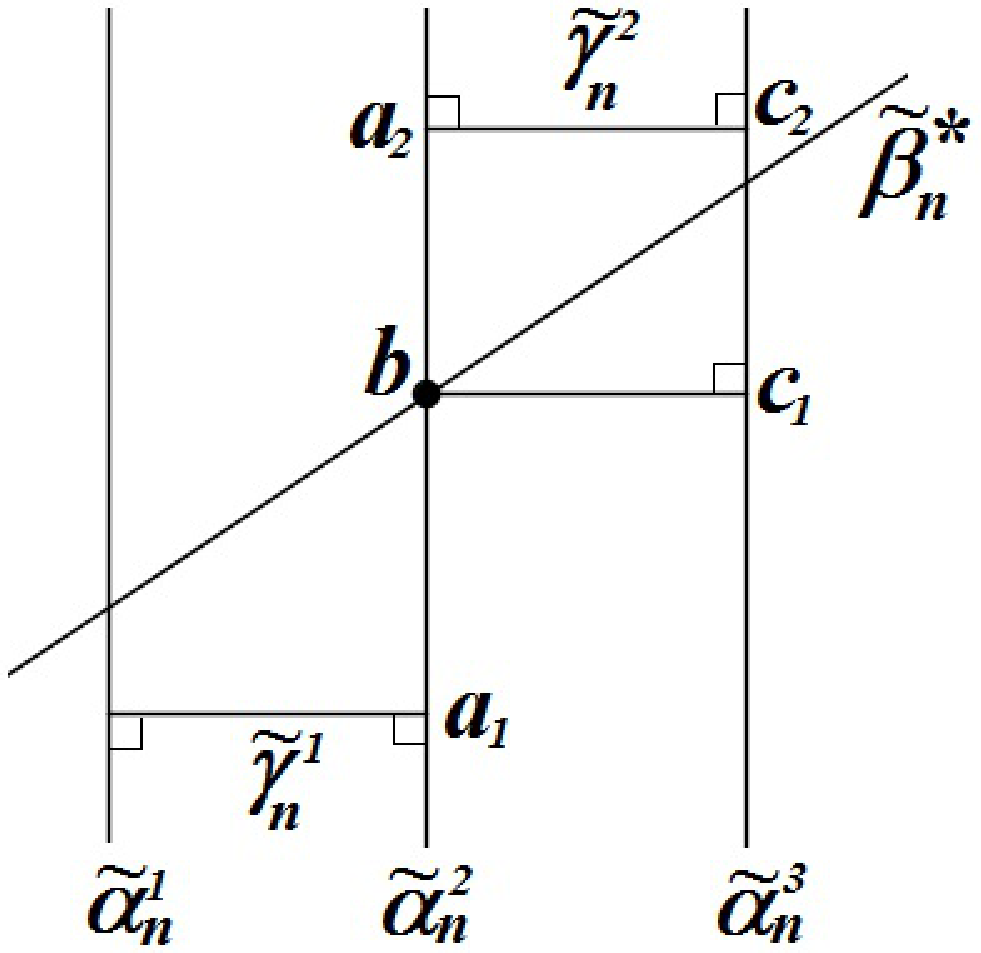}
\caption{}
\end{figure}

By the Collar lemma \cite{Buser}, there exists $C_1(M)>0$ such that
\begin{equation}
\label{eq:***}
l_{\tilde{\gamma}_n^i}(X)=l_{\tilde{\gamma}_n^i}(X')\geq C_1(M)\max \{ 1,|\log l_{\alpha_n}(X_0)|\}.
\end{equation}

Since $\sinh d_{hyp}(b,c_1)\leq \frac{e^{d_{hyp}(b,c_1)}}{2}$ and $\cosh d_{hyp}(b,a_2)\geq\frac{e^{d_{hyp}(b,a_2)}}{2}$, and by (\ref{eq:**}) and (\ref{eq:***}), we have
\begin{equation*}
d_{hyp}(b,a_2)\leq C_2(M)\max \{ 1,|\log l_{\alpha_n}(X_0)|\}+d_{hyp}(b,c_1)
\end{equation*}
which implies
\begin{equation}
\label{eq:4}
\begin{split}
d_{hyp}(b,a_2)\leq C_2(M)\max \{ 1,|\log l_{\alpha_n}(X_0)|\}+l_{\beta_n}(X)\leq\\ \leq
C_2(M)\max \{ 1,|\log l_{\alpha_n}(X_0)|\}+e^Ml_{\beta_n}(X_0).
\end{split}
\end{equation}

Note that by construction $$l_{\beta_n}(X_0)\leq l_{\gamma_n^1}(X_0)+l_{\gamma_n^2}(X_0)+l_{\alpha_n}(X_0).$$
We estimate $l_{\gamma_n^i}(X_0)$ from the above using right-angled pentagons.
Namely
each hexagon of $P_i^0$ contains a half of the arc $\gamma^i_n$ and  $\gamma^i_n$ intersects the b-side of the hexagons that connects the two boundary geodesics of $P_i^0$ different from $\alpha_n$. Then both hexagons of $P_i^0$ are divided into two right-angled pentagons by $\gamma^i_n$. 
The sides of the obtained pentagons are as follows in the cyclic order: a portion of the a-side on $\alpha_n$, followed by the half of $\gamma_n^i$, followed by a portion of a b-side, followed by an a-side on a boundary curve different from $\alpha_n$ and followed by a b-side. 
We choose one of the two pentagons such that the portion of the a-side has length at least $\frac{1}{4}l_{\alpha_n}(X_0)$.
Since $l_{\alpha_n}(X_0)\leq M_0$ it follows that any a-side of the hexagon has length at most $\frac{1}{2}M_0$ and a hyperbolic formula for the right-angled pentagons gives
$$
\cosh \frac{M_0}{2}\geq\sinh \frac{1}{4}l_{\alpha_n}(X_0)\sinh\frac{1}{2}l_{\gamma_n^i}(X_0)\geq\frac{1}{4}l_{\alpha_n}(X_0)\sinh\frac{1}{2}l_{\gamma_n^i}(X_0).
$$
The inequality implies that there is $C_3(M_0)$ such that
$$
l_{\gamma_n^i}(X_0)\leq C_3(M_0)\max\{ 1,|\log l_{\alpha_n}(X_0)|\}
$$
which in turn implies
\begin{equation}
\label{eq:L}
l_{\beta_n}(X_0)\leq C_4(M_0)\max\{ 1,|\log l_{\alpha_n}(X_0)|\}
\end{equation}
for some constant $C_4(M_0)>0$.

By (\ref{eq:*}), (\ref{eq:4}) and (\ref{eq:L}) we have
\begin{equation}
\label{eq:twist-upper-bound}
\frac{|t_{\alpha_n}(X)|}{2}\leq d_{hyp}(a_2,b)\leq [C_2(M)+e^MC_4(M_0)]\\max\{ 1,|\log l_{\alpha_n}(X_0)|\} 
\end{equation}
which gives 
\begin{equation}
\label{eq:5}
\frac{|t_{\alpha_n}(X)|}{\max\{ 1,|\log l_{\alpha_n}(X_0)|\}}\leq C_5(M,M_0)
\end{equation}
and this finishes the proof of $F(T_{ls}(X_0))\subset l^{\infty}$ in the case $P_1^0\neq P_2^0$.

{\bf Case 2.} Assume that $P_1^0=P_2^0$. We define $\gamma_n$ to be the unique geodesic arc in $f(P_1^0)=f(P_2^0)=P_1'$ starting and ending at $\alpha_n$
and orthogonal to $\alpha_n$ at both of its endpoints. Then we define a closed curve $\beta_n$ to consists of $\gamma_n$ and an arc of $\alpha_n$ of the size $t_{\alpha_n}(X')$. The above argument applies to this case as well.

\vskip .1 cm

\noindent {\bf Step II: $F:T_{ls}(X_0)\to l^{\infty}$ is surjective.} For $a\in l^{\infty}$ the surface $X_a$ obtained by gluing the pants with prescribed cuffs and twists is complete. The marking map for $X_a$ can be chosen to be a homeomorphism because each twist is realized in an annulus containing a given cuff. Let $X_a'=X_{a'}$ be the surface obtained by a $K$-quasiconformal map $f:X_0\to X_a'$ such that $l_{\alpha_n}(X_a')=l_{\alpha_n}(X_a)$ for all $n\in\mathbb{N}$ as before (cf. \cite{Bishop}). 
Then
\begin{equation*}
\Big{|}\log\frac{l_{\beta}(X_a')}{l_{\beta}(X_0)}\Big{|}\leq M(K)<\infty
\end{equation*}
for $n\in\mathbb{N}$. The surface $X_a$ is obtained by a multi twist around $\mathcal{P}=\{\alpha_n\}$ by the amount $\{ t_n=t_{\alpha_n}(X_a)-t_{\alpha_n}(X_a')\}$. Note that $0\leq t_{\alpha_n}(X_{a'})=\frac{l_{\alpha_n}(X_a)}{l_{\alpha_n}(X_0)}t_{\alpha_n}(X_0)<l_{\alpha_n}(X_a)$. By the Collar lemma \cite{Buser}, we have that 
\begin{equation}
\label{eq:str}
l_{\beta}(X_{a'}),l_{\beta}(X_a)\geq C\sum_n i(\alpha_n,\beta )\max\{ 1,|\log l_{\alpha_n}(X_0)|\}
\end{equation}
for each simple closed curve $\beta$ on $X_0$. Then
$$
\log\frac{l_{\beta}(X_a)}{l_{\beta}(X_{a'})}\leq \log\frac{l_{\beta}(X_{a'})+\sum_ni(\alpha_n,\beta )|t_n|}{l_{\beta}(X_{a'})}
$$
and
$$
|t_n|\leq C'\max\{ 1,|\log l_{\alpha_n}(X_0)|\}
$$
which together with (\ref{eq:str}) implies that
$$
\log\frac{l_{\beta}(X_a)}{l_{\beta}(X_{a'})}\leq \frac{C'}{C}=C''
$$
for each $\beta\in\mathcal{C}$. In the exactly the same fashion, we obtain
$$
\log\frac{l_{\beta}(X_{a'})}{l_{\beta}(X_{a})}\leq C'''
$$
for each $\beta\in\mathcal{C}$. Thus $X_a\in T_{ls}(X_0)$ and $F:T_{ls}(X_0)\to l^{\infty}$ is onto.
\vskip .1 cm

\noindent {\bf Step III: $F:T_{ls}(X_0)\to l^{\infty}$ is localy Lipschitz.} Let $X_1\in T_{ls}(X_0)$ be fixed and let $X,Y\in B_{\frac{1}{2}}(X_1)$ be two arbitrary points in the ball of radius $\frac{1}{2}$ centered at $X_1\in T_{ls}(X_0)$. Consequently $d_{ls}(X,Y)<1$. 
Note that $|\log\frac{l_{\alpha_n}(X)}{l_{\alpha_n}(X_0)}- \log\frac{l_{\alpha_n}(Y)}{l_{\alpha_n}(X_0)}|=|\log\frac{l_{\alpha_n}(X)}{l_{\alpha_n}(Y)}|\leq d_{ls}(X,Y)$ for each $n\in\mathbb{N}$. It remains to consider $\frac{|t_{\alpha_n}(X)-t_{\alpha_n}(Y)|}{\max\{ 1,|\log l_{\alpha_n}(X_0)|\}}$. By \cite{ALPS} and \cite{Bishop}, there exists a $[1+Cd_{ls}(X,Y)]$-quasiconformal map $f:X\to X'$ such that $l_{\alpha_n}(X')=l_{\alpha_n}(Y)$ for each $n\in\mathbb{N}$ with $C=C(e^{d_{ls}(X_0,X_1)+1}M_0)>0$. 
Let $0\leq \tilde{t}_{\alpha_n}(X)<l_{\alpha_n}(X)$ be such that there exists an integer $k\in\mathbb{Z}$ with 
\begin{equation}
\label{eq:6}
t_{\alpha_n}(X)=k\cdot l_{\alpha_n}(X)+  \tilde{t}_{\alpha_n}(X).
\end{equation}
Note that, for $C'=C'(M_1+1,M_0)$, we have 
\begin{equation}
\label{eq:7}
|k|\leq C'\frac{ |\log l_{\alpha_n}(X_0)|}{l_{\alpha_n}(X_0)}
\end{equation}
by (\ref{eq:5}) and (\ref{eq:6}) because normalized twists $\tilde{t}_{\alpha_n}$ are bounded from the above and
$$
e^{-(M_1+1)}l_{\alpha_n}(X_0)\leq l_{\alpha_n}(X)\leq e^{M_1+1}l_{\alpha_n}(X_0).
$$

 The construction of $f:X\to X'$ from \cite{ALPS} implies that 
$$
t_{\alpha_n}(X')=k\cdot l_{\alpha_n}(Y)+ \frac{l_{\alpha_n}(Y)}{l_{\alpha_n}(X)} \tilde{t}_{\alpha_n}(X).
$$
We estimate $|t_{\alpha_n}(X')-t_{\alpha_n}(X)|$.  Let $M_1=d_{ls}(X_0,X_1)$. Then we have
\begin{equation}
\label{eq:diff-qc-twist}
\begin{split}
|t_{\alpha_n}(X')-t_{\alpha_n}(X)|\leq |k|\cdot |l_{\alpha_n}(Y)-l_{\alpha_n}(X)|+\tilde{t}_{\alpha_n}(X)\Big{|}\frac{l_{\alpha_n}(Y)}{l_{\alpha_n}(X)}-1\Big{|}\\
\leq C'\frac{\max\{ 1,|\log l_{\alpha_n}(X_0)|\}}{l_{\alpha_n}(X_0)}l_{\alpha_n}(X)\Big{|}\frac{l_{\alpha_n}(Y)}{l_{\alpha_n}(X)}-1\Big{|}+\tilde{t}_{\alpha_n}(X)\Big{|}\frac{l_{\alpha_n}(Y)}{l_{\alpha_n}(X)}-1\Big{|}
\end{split}
\end{equation}
which implies
\begin{equation}
\label{eq:norm-diff-qc-twist}
\frac{|t_{\alpha_n}(X)-t_{\alpha_n}(X')|}{\max \{ 1,|\log l_{\alpha_n}(X_0)|\}}\leq (C' e^M+M_0e^M)\Big{|}\frac{l_{\alpha_n}(Y)}{l_{\alpha_n}(X)}-1\Big{|}\leq C_1\Big{|}\log\frac{l_{\alpha_n}(Y)}{l_{\alpha_n}(X)}\Big{|}.
\end{equation}

Note that $d_{ls}(X,X')\leq C d_{ls}(X,Y)$ which implies that 
$$
d_{ls}(X',Y)\leq d_{ls}(X',X)+d_{ls}(X,Y)\leq C_2 d_{ls}(X,Y).
$$
 
 Let
\begin{equation}
\label{eq:relative_t_n}
t_n=t_{\alpha_n}(Y)-t_{\alpha_n}(X').
\end{equation}
The surface $Y$ is obtained from $X'$ by a multi twist along $\{\alpha_n\}$ by the amount $\{ t_n\}$. As before, we divide the argument into two cases: $P_1'=P_2'$ and $P_1'\neq P_2'$.

{\bf Case 1.}
Given $\alpha_n$, assume first that the pairs of pants with boundary $\alpha_n$ are equal, namely $P_1'=P_2'$. Let $\beta_n$ be a closed curve obtained by concatenating the unique arc $\gamma_n$ in $P_1'$ orthogonal to $\alpha_n$ at both of its endpoints followed by an arc on $\alpha_n$ of  the length $\tilde{t}_{\alpha_n}(X')$. Let $\beta_n^{*}$ be the geodesic representative of $\beta_n$. 

Denote by $X_t'$ the hyperbolic surface obtained by twisting the amount $t\cdot t_n$, for $t\in\mathbb{R}$ and $t_n$ defined by (\ref{eq:relative_t_n}), along the cuffs $\alpha_n$ on the surface $X'$. Note that $X_0'=X'$ and that by the definition of $t_n$ we have that $X_1'=Y$.

Recall that (cf. \cite{Ker}, \cite{EpsteinMarden})
\begin{equation}
\label{eq:deriv-length}
\frac{d}{d(t\cdot t_n)}l_{\beta_n^{*}}(X_t')=\cos \varphi_t^{*}
\end{equation}
where $\varphi_t^{*}\in (0,\pi )$ is the angle between $\tilde{\beta}^{*}_n$ and $\tilde{\alpha}_n$.
Let us fix $\epsilon_0>0$. Note that $\varphi_t^{*}$ is either increasing or decreasing from $\varphi_0^{*}$ to $\varphi_1^{*}$ in $t$ for $0\leq t\leq 1$ (depending whether $t_n$ is positive or negative) due to the fact that the geodesic length along a left earthquake with support $\alpha_n$ is a convex function (cf. \cite{Ker}, \cite{EpsteinMarden}).  

Assume that $t_n>0$. If $\cos\varphi_0^{*}\geq\epsilon_0$ (which implies $\cos\varphi_t^{*}\geq\epsilon_0$ for $0\leq t$) then we set $\beta_n^{**}:=\beta_n^{*}$. If $\cos\varphi_0^{*}<\epsilon_0$ then we choose $\beta_n^{**}$ such that $\cos\varphi_t^{**}>\epsilon_0$ as follows.

Consider universal covering $\pi:\mathbb{H}^2\to X'$ such that one lift $\tilde{\alpha}_n$ of $\alpha_n$ is the $y$-axis. Further we arrange that two lifts $\tilde{\gamma}_n^{-1}$ and $\tilde{\gamma}_n^1$ of the arc $\gamma_n$ that are adjacent to the $y$-axis from the left and the from the right meet the $y$-axis between $i$ and $e^{l_{\alpha_n}(X')}i$. Let $b<0$ be an endpoint on $\mathbb{R}$ of the hyperbolic geodesic containing $\tilde{\gamma}_n^{-1}$ and let $a>0$ be an endpoint on $\mathbb{R}$ of the geodesic containing $\tilde{\gamma}_n^1$. For any $k\in\mathbb{Z}$, a $k$ full left twists on $\alpha_n$ on the surface $X'$ maps the curve $\beta_n^{*}$ to a new curve $\beta_n^{**}$. The curve obtained by the concatenating the arc $\gamma_n$ with the arc which winds around $\alpha_n$ $k$-times plus the shear amount $\tilde{t}_{\alpha_n}(X')$ is homotopic to $\beta_n^{**}$. The lift of the above arc has two orthogonal sub arcs to the $y$-axis one from the left which is equal to $\tilde{\gamma}_n^{-1}$ which meets $y$ axis at a point $|b|i$ between $i$ and $e^{l_{\alpha_n}(X')}i$, and the other orthogonal arc $\tilde{\gamma}_n^2$ which meets the $y$-axis at a point $c_2=|a|e^{kl_{\alpha_n}(X')}i$. By the definition of left twists,
it follows that one endpoint of a lift $\tilde{\beta}_n^{**}$ of $\beta_n^{**}$ is between $b$ and $0$, and the other endpoint of $\tilde{\beta}_n^{**}$ is between $ae^{kl_{\alpha_n}(X')}$ and $\infty$. Among all the geodesics whose one endpoint is in the interval $[b,0)$ and the other endpoint is in the interval $[ae^{kl_{\alpha_n}(X')},\infty )$, the geodesic with endpoints $b$ and $ae^{kl_{\alpha_n}(X')}$ subtends the largest angle $\varphi_0$ with the $y$-axis. We have
$$
\cos\varphi_0 =\frac{ae^{kl_{\alpha_n}(X')}+b}{ae^{kl_{\alpha_n}(X')}-b}.
$$
Define
$$
k=\big{[}\frac{1}{l_{\alpha_n}(X')}\log\frac{1+\epsilon_0}{1-\epsilon_0}\big{]}+2
$$
where $[x]$ is the integer part of $x\in\mathbb{R}$. Then we have that
$$
\cos\varphi_0\geq\epsilon_0
$$
which implies that
$$
\frac{1}{t_n}\frac{d}{dt}l_{\beta_n^{**}}(X_t')=\frac{d}{d(t\cdot t_n)}l_{\beta_n^{**}}(X_t')\geq \epsilon_0
$$
for all $t\in [0,1]$.

Note that
$$
l_{\beta_n^{**}}(X')\leq kl_{\alpha_n}(X')+C|\log l_{\alpha_n}(X')|\leq
C'\max\{ 1,|\log l_{\alpha_n}(X_0)|\}
$$

By the Mean Value Theorem there exists $t^{*}\in (0,1)$ such that
$$
|l_{\beta_n^{**}}(Y)-l_{\beta_n^{**}}(X')|=|\frac{d}{dt}l_{\beta_n^{**}}(X'_{t^{*}})|\geq \epsilon_0t_n
$$
because $X'_1=Y$. Since $l_{\beta_n^{**}}(X')\leq C'\max\{ 1,|\log l_{\alpha_n}(X_0)|\}$, the above gives
$$
\frac{|t_n|}{\max\{ 1,|\log l_{\alpha_n}(X_0)|\}}\leq\frac{|t_n|}{l_{\beta_n^{**}}(X')}\leq\frac{C}{\epsilon_0}|\frac{l_{\beta_n^{**}}(Y)}{l_{\beta_n^{**}}(X')}-1|\leq\frac{C}{\epsilon_0}|\log \frac{l_{\beta_n^{**}}(Y)}{l_{\beta_n^{**}}(X')}|.
$$

\vskip .1 cm

Assume now that $t_n<0$. Then we use a similar method by considering $\cos \varphi_t^{*}\leq -\epsilon_0$ and $k$ full right twist around $\alpha_n$ to replace $\beta_n^{*}$ with $\beta_n^{**}$. The proof proceeds analogously. 

\vskip .2 cm

{\bf Case 2.} The second case is when $P_1'\neq P_2'$. 
Define a closed curve $\beta_n\subset P_1'\cup P_2'\subset X'$ to consists of the 
unique arc $\gamma_n^1$ in $P_1'$ orthogonal at both of its endpoints to $
\alpha_n$ followed by the arc in $\alpha_n$ (in the direction of the left twist) of the 
size at most $l_{\alpha_n}(X')$ followed by the unique arc $\gamma_n^2\subset 
P_2'$ orthogonal to $\alpha_n$ at both of its endpoints followed by an arc on $
\alpha_n$ of size at most  $l_{\alpha_n}(X')$. 
For the convenience of the notation, denote by $\beta_n$ the closed geodesic 
homotopic to $\beta_n$.  
The arcs $\gamma_n^i$, for $i=1,2$, have lengths comparable to $\max\{ 1,|\log 
l_{\alpha_n}(X_0)|\}$ up to positive multiplicative constants. 

Let $\tilde{\alpha}_n^j$, for $j=1,2$, be two consecutive lifts of $\alpha_n$. Two lifts $\tilde{\gamma}_n^{j,k}$, for $k=1,2$, of $\gamma_n^j$ which meet $\tilde{\alpha}_n^j$ can be chosen such that the distance between their foots on $\tilde{\alpha}_n^j$ is at most $l_{\alpha_n}(X')$. Assume that $t_n>0$. We perform $k$ full left twists along $\alpha_n$ to obtain a new closed curve $\beta_n^{**}$ from the closed curve $\beta_n^{*}$. When $k=[\frac{1}{l_{\alpha_n}(X')}\log\frac{1+\epsilon_0}{1-\epsilon_0}]+2
$, we get (similar to Case 1) for both angles $\varphi_n^j$ that the new closed geodesic $\beta_n^{**}$ subtends with $\alpha_n$,
$$
\cos\varphi_n^j\geq \epsilon_0.
$$
Then
$$
\frac{d}{dt}l_{\beta_n^{**}}(X'_t)=\cos\varphi_n^1+\cos\varphi_n^2\geq 2\epsilon_0
$$
which gives
$$
\frac{|t_n|}{\max\{ 1,|\log l_{\alpha_n}(X_0)|\}}\leq C|\log\frac{l_{\beta_n^{**}}(Y)}{l_{\beta_n^{**}}(X')}|.
$$
When $t_n<0$, the proof proceeds as before.

Thus we established that the map $F:T_{ls}(X_0)\to l^{\infty}$ is locally Lipschitz.

\vskip .1 cm

\noindent {\bf Step IV: $F^{-1}:l^{\infty}\to T_{ls}(X_0)$ is locally Lipschitz.}
We consider the map $F^{-1}:l^{\infty}\to T_{ls}(X_0)$ and prove that it is also locally Lipschitz. Let $a_0\in l^{\infty}$ be fixed. Denote by $X_{a_0}$ the surface corresponding to $a_0$, namely $X_{a_0}=F^{-1}(a_0)\in T_{ls}(X_0)$. Let $a,b\in l^{\infty}$ such that $\| a-a_0\|_{\infty}< \frac{1}{2}$ and $\| b-a_0\|_{\infty} <\frac{1}{2}$ which implies $\| a-b\|_{\infty}<1$. 
There exists a $(1+C|\log\frac{l_{\alpha_n}(X_b)}{l_{\alpha_n}(X_a)}|)$-quasiconformal map $f: X_b\to X_b'$ such that $\l_{\alpha_n}(X_b')=l_{\alpha_n}(X_a)$ for all $n$ (cf. \cite{Bishop}, \cite{ALPS}). 
Recall that
$$
t_{\alpha_n}(X_b)=k l_{\alpha_n}(X_b)+\tilde{t}_{\alpha_n}(X_b)
$$
where $k\in\mathbb{Z}$, $0\leq \tilde{t}_{\alpha}(X_b)< l_{\alpha_n}(X_b)$ and
\begin{equation}
\label{eq:star}
|k|\leq\frac{C\max\{ 1,|\log l_{\alpha_n}(X_0)|\}}{l_{\alpha_n}(X_b)}.
\end{equation}

By the construction of $f:X_b\to X_b'$, we have 
$$
t_{\alpha_n}(X_b')=kl_{\alpha_n}(X_a)+\frac{l_{\alpha_n}(X_a)}{l_{\alpha_n}(X_b)}\tilde{t}_{\alpha_n}(X_b).
$$

It follows that
\begin{equation}
\label{eq:diff-in-twists}
|t_{\alpha_n}(X_b)-t_{\alpha_n}(X_b')|\leq |k|l_{\alpha_n}(X_a)\Big{|}\frac{l_{\alpha_n}(X_b)}{l_{\alpha_n}(X_a)}-1\Big{|}+l_{\alpha_n}(X_b)\Big{|}\frac{l_{\alpha_n}(X_b)}{l_{\alpha_n}(X_a)}-1\Big{|}.
\end{equation}
Since $a,b\in l^{\infty}$, it follows that there exists $C>0$ such that
\begin{equation}
\label{eq:twostars}
\Big{|}\frac{l_{\alpha_n}(X_b)}{l_{\alpha_n}(X_a)}-1\Big{|}\leq C\Big{|}\log \frac{l_{\alpha_n}(X_b)}{l_{\alpha_n}(X_a)}\Big{|}.
\end{equation}
The inequalities (\ref{eq:diff-in-twists}), (\ref{eq:star}) and (\ref{eq:twostars}) imply
$$
|t_{\alpha_n}(X_b)-t_{\alpha_n}(X_b')|\leq C\max\{ 1,|\log l_{\alpha_n}(X_0)|\}\frac{l_{\alpha_n}(X_a)}{l_{\alpha_n}(X_b)}|\log \frac{l_{\alpha_n}(X_b)}{l_{\alpha_n}(X_a)}|
+C'|\log \frac{l_{\alpha_n}(X_b)}{l_{\alpha_n}(X_a)}|
$$
where $C'=C'(\| a_0\|_{\infty} +\frac{1}{2})$, and since $|\log \frac{l_{\alpha_n}(X_b)}{l_{\alpha_n}(X_a)}|\leq \| a-b\|_{\infty}$, we get
\begin{equation}
\label{eq:irr-twists-b-b'}
\frac{|t_{\alpha_n}(X_b)-t_{\alpha_n}(X_b')|}{\max\{ 1,|\log l_{\alpha_n}(X_0)|\}}\leq C''\| a-b\|_{\infty}.
\end{equation}
 
Since $f:X_b\to X_b'$ is a $(1+C\| a-b\|_{\infty})$-quasiconformal, it follows that $d_{ls}(X_b,X_b')\leq C\| a-b\|_{\infty}$. Moreover, if $X_b'=F^{-1}(b')$ then (\ref{eq:irr-twists-b-b'}) implies that that $\| b'-b\|_{\infty}\leq C\| a-b\|_{\infty}$. Finally, $\| a-b'\|_{\infty}\leq \|a-b\|_{\infty}+\|b-b'\|_{\infty}\leq C\|a-b\|_{\infty}$.

It remains to estimate the length-spectrum distance between $X_b'=X_{b'}$ and $X_a$. This part of the argumentt is essentially contained in \cite{ALPS}.
Note that $X_a$ is obtained from $X_{b'}$ by multi twist along $\alpha_n$ by the amount $t_n'=t_{\alpha_n}(X_a)-t_{\alpha_n}(X_{b'})$. The estimate (\ref{eq:irr-twists-b-b'}) and the triangle inequality $\| t_{\alpha_n}(X_a)-t_{\alpha_n}(X_{b'})\|_{\infty}\leq \| t_{\alpha_n}(X_a)-t_{\alpha_n}(X_{b})\|_{\infty} +\| t_{\alpha_n}(X_b)-t_{\alpha_n}(X_{b'})\|_{\infty}$ gives that
$$
|t_n'|=|t_{\alpha_n}(X_a)-t_{\alpha_n}(X_{b'})|\leq C\|a-b\|_{\infty}\max\{ 1,|\log l_{\alpha_n}(X_0)|\}.
$$

For any simple closed geodesic $\beta$ on $X_{b'}$, we estimate $|\log \frac{l_{\beta}(X_{b'})}{l_{\beta}(X_a)}|$. We have
\begin{equation*}
\begin{split}
l_{\beta}(X_{b'})\leq l_{\beta}(X_a)+\sum_{n=1}^{\infty}i(\beta ,\alpha_n)|t_n'|
\leq l_{\beta}(X_a)+\\ +C\|a-b\|_{\infty}\sum_{n=1}^{\infty}i(\beta ,\alpha_n)\max\{ 1,|\log l_{\alpha_n}(X_0)|\}
\end{split}
\end{equation*}
and
$$
l_{\beta}(X_a)\geq C'\sum_{n=1}^{\infty} i(\beta ,\alpha_n)\max\{ 1,|\log l_{\alpha_n}(X_a)|\}
$$
by the Collar lemma. Since $X_a\in T_{ls}(X_0)$, it follows that
there exists $M>0$ such that $|\log l_{\alpha_n}(X_a)|\geq |\log l_{\alpha_n}(X_0)|-M$. Thus there exists $C''>0$ such that
$$
\max\{ 1,|\log l_{\alpha_n}(X_a)|\}\geq C''\max\{ 1,|\log l_{\alpha_n}(X_0)|\}.
$$
The above inequalities imply that
$$
\frac{l_{\beta}(X_{b'})}{l_{\beta}(X_a)}\leq 1+C'''\| a-b\|_{\infty}
$$
and by reversing roles played by $X_a$ and $X_{b'}$ we get
$$
\frac{l_{\beta}(X_{a})}{l_{\beta}(X_{b'})}\leq 1+C'''\| a-b\|_{\infty}.
$$
This proves that $F^{-1}:l^{\infty}\to T_{ls}(X_0)$ is Lipschitz.
\end{proof}

Since $l^{\infty}$ is contractible, we get

\begin{corollary}
\label{cor:contractible}
The length spectrum Teichm\"uller space $T_{ls}(X_0)$ for any hyperbolic surface $X_0$ with an upper bounded pants decomposition is contractible.
\end{corollary}

\section{The closure of $T_{qc}(X_0)$ in $T_{ls}(X_0)$}

A question of characterizing the closure of the image of $T_{qc}(X_0)$ inside $T_{ls}(X_0)$ was raised in \cite{ALPS}. We use our understanding of the topology on the Fenchel-Nielsen coordinates that makes the map $F:T_{ls}(X_0)\to l^{\infty}$ into a homeomorphism to give a characterization of the closure of $T_{qc}(X_0)$.

Let $l=\{ (x_1,x_2,\ldots ):x_i\in\mathbb{R}\}$ be the space of all sequences of real numbers.
We first define $\tilde{F}:T_{ls}(X_0)\to l$ by setting
\begin{equation*}
\tilde{F}(X)=\{ (x_1,x_2,\ldots )\in l:x_{2n-1}=\log \frac{l_{\alpha_n}(X)}{l_{\alpha_n}(X_0)},\ x_{2n}=t_{\alpha_n}(X)-t_{\alpha_n}(X_0)\mbox{ for } n\in\mathbb{N}\}.
\end{equation*}
If $\alpha_n$ is a boundary component we use only the length coordinate.

By \cite{ALPS} or by Theorem \ref{eq:fn-norm}, $\tilde{F}(T_{ls})\subset l$ consists of all $\bar{x}=(x_1,x_2,\ldots )\in l$ such that
$$
\sup_n\max\{ |x_{2n-1}|,\frac{|x_{2n}|}{\max\{ 1,|\log l_{\alpha_n}(X_0)|\} }\} <\infty .
$$

Let $O(1)$ denotes a bounded function and let $O(M):=M\cdot O(1)$ as $M\to\infty$. Moreover, $o(1)$ denotes a function which converges to $0$ as $M\to\infty$ and let $o(M)=M\cdot o(1)$.
Then $\bar{x}=(x_1,x_2,\ldots )$ are the Fenchel-Nielsen coordinates of $X\in T_{ls}(X_0)$ if and only if 
$$|x_{2n-1}|=O(1)$$ and 
$$|x_{2n}|=O(\max\{ 1,|\log l_{\alpha_n}(X_0)|\}).
$$
By \cite{ALPSS}, the image $F(T_{qc}(X_0))\subset l$ of the quasiconformal Teichm\"uller space $T_{qc}(X_0)$ consists of all $\bar{x}=(x_1,x_2,\ldots )$ such that 
$$\| \bar{x}\|_{\infty}<\infty ,$$
or equivalently
$$
|x_n|=O(1).
$$

\begin{theorem}
\label{thm:closure-qc}
Let $X_0$ be a complete hyperbolic surface with an upper bounded pants decomposition $\mathcal{P} =\{\alpha_n\}$. Then $X\in T_{ls}(X_0)$ is in the closure of $T_{qc}(X_0)$ for the metric $d_{ls}$ if and only if
\begin{equation*}
\sup_{\alpha_n\in\mathcal{P}}\Big{|}\log\frac{l_{\alpha_n}(X)}{l_{\alpha_n}(X_0)}\Big{|}<\infty
\end{equation*}
and
\begin{equation}
\label{eq:o(length)}
|t_{\alpha_n}(X)-t_{\alpha_n}(X_0)|=o(|\log l_{\alpha_n}(X_0)|)
\end{equation} 
as $|\log l_{\alpha_n}(X_0)|\to\infty$.
\end{theorem}

\begin{proof}
We first note that if $X_0$ has a (geodesic) pants decomposition which
is bounded from the above and from the below then (cf.
\cite{Shiga}, \cite{ALPS}) $T_{qc}(X_0)=T_{ls}(X_0)$. Therefore we assume that there is a pants decomposition of $X_0$ which is upper bounded with a sequence of cuffs whose lengths go to $0$.

Let $X_i\in T_{qc}(X_0)$ such that $X_i\to X$ in the length spectrum metric $d_{ls}$ as $i\to\infty$. Then $d_{ls}(X_0,X)<\infty$, namely $X\in T_{ls}(X_0)$. Let $\{\alpha_{n_k}\}_k$ be the set of all geodesics in $\mathcal{P}$ such that $l_{\alpha_{n_k}}(X_0)\leq\frac{1}{e}$. Then by Theorem \ref{thm:main}
$$
\sup_k\frac{|t_{\alpha_{n_k}}(X_i)-t_{\alpha_{n_k}}(X)|}{|\log l_{\alpha_{n_k}}(X_0)|}\to 0
$$
as $i\to\infty$. Thus for any $\epsilon >0$ there exists $i_0$ such that for all $i>i_0$ we have
$$
|t_{\alpha_{n_k}}(X)-t_{\alpha_{n_k}}(X_0)|\leq |t_{\alpha_{n_k}}(X_i)-t_{\alpha_{n_k}}(X_0)|+\epsilon |\log l_{\alpha_{n_k}}(X_0)|.
$$
Assume on the contrary that (\ref{eq:o(length)}) is false. Then there exists $C>0$ and  subsequence $k_j$ such that $l_{\alpha_{n_{k_j}}}(X_0)\to 0$ as $j\to\infty$ and
$$
|t_{\alpha_{n_{k_j}}}(X)-t_{\alpha_{n_{k_j}}}(X_0)|\geq C|\log l_{\alpha_{n_{k_j}}}(X_0)|.
$$
Choose $\epsilon =\frac{C}{2}$. The above two inequalities give for all $i>i_0$
$$
 |t_{\alpha_{n_{k_j}}}(X_{i})-t_{\alpha_{n_{k_j}}}(X_0)|\geq \frac{C}{2}|\log l_{\alpha_{n_{k_j}}}(X_0)|
$$
which contradicts $X_{i}\to X$ as $i\to\infty$. Thus $X$ satisfies (\ref{eq:o(length)}).

\vskip .1 cm

Assume that $X\in T_{ls}(X_0)$ satisfies (\ref{eq:o(length)}). We need to find a sequence $X_i\in T_{qc}(X_0)$ such that $X_i\to X$ as $i\to\infty$ for the length spectrum metric $d_{ls}$. 
For a given $i\in\mathbb{N}$, let $X_i\in T_{ls}(X_0)$ be defined by the Fenchel-Nielsen coordinates
$$
l_{\alpha_n}(X_i):=l_{\alpha_n}(X)
$$
and 
\begin{equation}
t_{\alpha_n}(X_i)-t_{\alpha_n}(X_0):=\mbox{sgn} [t_{\alpha_n}(X)-t_{\alpha_n}(X_0)] \min\{ |t_{\alpha_n}(X)- t_{\alpha_n}(X_0)|,i\}.
\end{equation}
By \cite{ALPSS}, we have $X_i\in T_{qc}(X_0)$. Let $M=d_{ls}(X_0,X)$ and choose $\epsilon >0$. Since $X$ satisfies (\ref{eq:o(length)}), it follows that there exists $\delta >0$ such that
$$
\frac{|t_{\alpha_n}(X)-t_{\alpha_n}(X_0)|}{|\log l_{\alpha_n}(X_0)|}<\frac{\epsilon}{2}
$$
for all $\alpha_n\in\mathcal{P}$ with $l_{\alpha_n}(X_0)\leq\delta$. Moreover, there exists $C=C(\delta )>0$ such that
$$
|t_{\alpha_n}(X)-t_{\alpha_n}(X_0)|\leq C
$$
for all $\alpha_n\in\mathcal{P}$ with $l_{\alpha_n}(X_0)>\delta$. 
 
For $l_{\alpha_n}(X_0)\leq\delta$, we have
\begin{equation*}
\begin{split}
\frac{|t_{\alpha_n}(X)-t_{\alpha_n}(X_i)|}{|\log l_{\alpha_n}(X_0)|}\leq 
\frac{|t_{\alpha_n}(X)-t_{\alpha_n}(X_0)|}{|\log l_{\alpha_n}(X_0)|}+ \frac{|t_{\alpha_n}(X_0)-t_{\alpha_n}(X_i)|}{|\log l_{\alpha_n}(X_0)|}\leq \\
\leq 2\frac{|t_{\alpha_n}(X)-t_{\alpha_n}(X_0)|}{|\log l_{\alpha_n}(X_0)|}<\epsilon .\ \ \ \ \ \ \ \ \ \ \ \ \ \ \ \ \ \ \ \ \ \ \ \ \ \ 
\end{split}
\end{equation*}
 
 For $l_{\alpha_n}(X_0)>\delta$, we have that $t_{\alpha_n}(X_i)=t_{\alpha_n}(X)$ for each $i>C$. Thus $X_i\to X$ as $i\to\infty$ in the length spectrum metric $d_{ls}$.
 \end{proof}

\end{document}